\documentclass{amsart}
\usepackage[utf8]{inputenc}
\usepackage{amssymb}
\usepackage{enumerate}

\newtheorem{thr}{Theorem}
\newtheorem{lem}{Lemma}
\newtheorem{df}{Definition}
\newtheorem{cor}{Corollary}

\newtheorem{rem}{Remark}

\makeatletter
\@addtoreset{equation}{section}

\newcommand{\bes}{\begin{displaymath}}
\newcommand{\ees}{\end{displaymath}}
\newcommand{\be}{\begin{equation}}
\newcommand{\ee}{\end{equation}}
\newcommand{\ba}{\begin{eqnarray}}
\newcommand{\ea}{\end{eqnarray}}
\newcommand{\bas}{\begin{eqnarray*}}
\newcommand{\eas}{\end{eqnarray*}}
\newcommand{\@Bbb}[1]{\ensuremath{\Bbb #1}}

\newcommand{\B}{{\@Bbb B}}
\newcommand{\C}{{\@Bbb C}}
\newcommand{\E}{{\@Bbb E}}
\newcommand{\F}{{\@Bbb F}}
\newcommand{\G}{{\@Bbb G}}
\renewcommand{\P}{{\@Bbb P}}

\newcommand{\Q}{{\@Bbb Q}}
\newcommand{\bQ}{{\@Bbb Q}}
\newcommand{\N}{{\@Bbb N}}
\newcommand{\R}{{\@Bbb R}}
\newcommand{\T}{{\@Bbb T}}
\newcommand{\bbR}{{\@Bbb R}}
\newcommand{\W}{{\@Bbb W}}
\newcommand{\Z}{{\@Bbb Z}}
\newcommand{\bbZ}{{\@Bbb Z}}

\newcommand{\@s}[1]{\ensuremath{\mathcal #1}}
\newcommand{\cA}{\@s A}
\newcommand{\cB}{\@s B}
\newcommand{\cC}{\@s C}
\newcommand{\cD}{\@s D}
\newcommand{\cE}{\@s E}
\newcommand{\cF}{\@s F}
\newcommand{\cG}{\@s G}
\newcommand{\cH}{\@s H}
\newcommand{\cI}{\@s I}
\newcommand{\cJ}{\@s J}

\newcommand{\cK}{\@s K}
\newcommand{\cL}{\@s L}
\newcommand{\cN}{\@s N}
\newcommand{\cM}{\@s M}
\newcommand{\cO}{\@s O}
\newcommand{\cP}{\@s P}
\newcommand{\cQ}{\@s Q}
\newcommand{\cR}{\@s R}
\newcommand{\cS}{\@s S}
\newcommand{\cT}{\@s T}
\newcommand{\cU}{\@s U}
\newcommand{\cV}{\@s V}
\newcommand{\cW}{\@s W}
\newcommand{\cX}{\@s X}
\newcommand{\cY}{\@s Y}
\newcommand{\cZ}{\@s Z}

\def\d{{\rm d}}

\def\LIM{{\rm LIM}}

\newcommand{\@bm}[1]{\ensuremath{\mathbf #1}}
\newcommand{\bma}{\@bm a}\newcommand{\bmA}{\@bm A}
\newcommand{\bmb}{\@bm b}\newcommand{\bmB}{\@bm B}
\newcommand{\bmc}{\@bm c}\newcommand{\bmC}{\@bm C}
\newcommand{\bmd}{\@bm d}\newcommand{\bmD}{\@bm D}
\newcommand{\bme}{\@bm e}
\newcommand{\bmf}{\@bm f}\newcommand{\bmF}{\@bm F}
\newcommand{\bmg}{\@bm g}\newcommand{\bmG}{\@bm G}
\newcommand{\bmh}{\@bm h}\newcommand{\bmH}{\@bm H}
\newcommand{\bmi}{\@bm i}\newcommand{\bmI}{\@bm I}
\newcommand{\bmj}{\@bm j}
\newcommand{\bmk}{\@bm k}\newcommand{\bmK}{\@bm K}
\newcommand{\bml}{\@bm l}
\newcommand{\bmm}{\@bm m}\newcommand{\bmM}{\@bm M}
\newcommand{\bmn}{\@bm n}
\newcommand{\bmo}{\@bm o}
\newcommand{\bmp}{\@bm p}
\newcommand{\bmq}{\@bm q}\newcommand{\bmQ}{\@bm Q}
\newcommand{\bmr}{\@bm r}
\newcommand{\bms}{\@bm s}\newcommand{\bmS}{\@bm S}
\newcommand{\bmt}{\@bm t}
\newcommand{\bmu}{\@bm u}\newcommand{\bmU}{\@bm U}
\newcommand{\bmw}{\@bm w}\newcommand{\bmW}{\@bm W}
\newcommand{\bmv}{\@bm v}\newcommand{\bmV}{\@bm V}
\newcommand{\bmx}{\@bm x}\newcommand{\bmX}{\@bm X}\newcommand{\bx}{\@bm x}
\newcommand{\bmy}{\@bm y}\newcommand{\bmY}{\@bm Y}\newcommand{\by}{\@bm y}
\newcommand{\bmz}{\@bm z}\newcommand{\bmZ}{\@bm Z}

\newcommand{\bmzero}{\@bm 0}

\newcommand{\@g}[1]{\ensuremath{\mathfrak #1}}
\newcommand{\gA}{\@g A}
\newcommand{\gD}{\@g D}
\newcommand{\gJ}{\@g J}
\newcommand{\gF}{\@g F}
\newcommand{\gM}{\@g M}
\newcommand{\gR}{\@g R}

\newcommand{\commentout}[1]{{}}

\numberwithin{equation}{section}

\begin{document}


\title[Ergodicity and CLT...]{Ergodicity and Central Limit Theorem for random interval homeomorphisms}
\author{Klaudiusz Czudek and Tomasz Szarek}
\address{Tomasz Szarek, Institute of Mathematics Polish Academy of Sciences, Abrahama 18, 81-967 Sopot, Poland}
\email{tszarek@impan.pl}
\address{Klaudiusz Czudek, Institute of Mathematics Polish Academy of Sciences
\'Sniadeckich 8,  00-656 Warszawa, Poland}
\email{klaudiusz.czudek@gmail.com}



\thanks{The research of Klaudiusz Czudek was supported by the Polish Ministry of Science and Higher Education "Diamond Grant" 0090/DIA/2017/46. Tomasz Szarek was supported by the
  Polish NCN grant 2016/21/B/ST1/00033.}
\begin{abstract} The central limit theorem for Markov chains generated by iterated function systems consisting of orientation--preserving homeomorphisms of the interval is proved. We study also ergodicity of such systems.
\end{abstract}

\maketitle

\section{Introduction}

Random dynamical systems in general and iterated function systems in particular have been extensively studied for many years (see \cite{Arnold_rds, Deroin, Kifer} and the references given there). This note is concerned with iterated function systems generated by orientation--preserving homeomorphisms on the interval $[0, 1]$.
It contains a simple proof of unique ergodicity on the open interval $(0, 1)$ for a wide class of iterated function systems. At first this phenomena was proved by L. Alsed\'a and M. Misiurewicz for some function systems consisting of piecewise linear homeomorphisms (see \cite{Alseda-Misiurewicz}). More general iterated function systems were considered by M. Gharaei and A. I. Homburg in \cite{Homburg}.
Recently D. Malicet obtained unique ergodicity as a consequence of the contraction principle for time homogeneous random walks on the topological group of homeomorphisms defined on the circle and interval  (see \cite{Malicet}). His proof, in turn, is based upon an invariance principle of A. Avila and M. Viana (see \cite{Avila_Viana}).

The second main objective of this note is to establish a quenched central limit theorem for random interval homeomorphisms. 
The proof is based on the Maxwell--Woodroofe approach for ergodic stationary Markov chains (see \cite{Maxwell-Woodroofe}) which generalises the martingale approximation method due to M.D. Gordin and B.A. Lif\v{s}hits (see \cite{Gordin_Lifsic}). Their result allows us to prove the central limit theorem for the stationary Markov chain (the annealed central limit theorem). On the other hand, using some coupling techniques we are able to evaluate the distance between the Fourier transform of the stationary and an arbitrary non--stationary Markov chain. Hence the quenched central limit theorems follows. Lately quenched central limit theorems have been proved  for various non--stationary Markov processes in \cite{Komorowski-Walczuk, Szarek, Stenflo} (see also \cite{Derrenic_Lin}).
 For more information we refer 
the readers to the book by T. Komorowski et al. \cite{Komorowski-book}, 
where a more detailed description of recent results is provided. Many results were formulated for Markov processes with transition probabilities satisfying the spectral gap property in the total variation norm or, at least, in the Kantorovich--Rubinstein norm. 
Since such processes are asymptotically stable and, in particular, have a unique invariant measure, our system does not satisfy this property. Indeed, $\delta_0$ and $\delta_1$ are its invariant measures too. 
 
 The paper is organised as follows. In Section 2 we introduce notation. Section 3 is devoted to the proof of the unique ergodicity and stability of iterated function systems on $(0, 1)$. Section 4 provides some auxiliary lemmas which are used in the proof of the main result (the quenched central limit theorem) in Section 5.

\section{Notation}

Let $(S, d)$ be a metric space. By $\mathcal{M}(S)$ we denote the set of all finite measures  on the $\sigma$-algebra $\mathcal{B}(S)$ of  all Borel subsets of $S$ and by $\mathcal{M}_1(S)\subseteq\mathcal{M}(S)$ we denote the subset of all probability measures on $S$. By $C(S)$ we denote the family of all bounded continuous functions equipped with the supremum norm $\|\cdot \|$. We shall write $\langle \mu, f\rangle$ for $\int_S f\d\mu$.

An operator $P : \mathcal{M}(S)\rightarrow\mathcal{M}(S)$ is called a {\it Markov operator} if it satisfies the following two conditions:
\begin{enumerate}[1)]
\item $P(\lambda_1\mu_1+\lambda_2\mu_2)=\lambda_1P\mu_1+\lambda_2P\mu_2$ for $\lambda_1, \lambda_2\geq0, \mu_1, \mu_2\in\mathcal{M}(S)$,
\item $P\mu(S)=\mu(S)$ for $\mu\in\mathcal{M}(S)$.
\end{enumerate}

A Markov operator $P$ is called a {\it Feller operator} if there is a linear operator $U : C(S)\rightarrow C(S)$ such that $U^*=P$, i.e.,
$$
\langle \mu, Uf\rangle=\langle P\mu, f\rangle\qquad \textrm{for $f\in C(S), \mu\in\mathcal{M}(S)$}.
$$

A measure $\mu_*$ is called {\it invariant} for a Markov operator $P$ if $P\mu_*=\mu_*$. If $S$ is a compact metric space, then every Feller operator $P$ has an invariant probability measure. For example, let $\mu\in\mathcal M_1(S)$ and define $\nu\in C(S)^*$ by $\nu(f)=\LIM(\frac{1}{n}\sum_{k=1}^n\langle P^k\mu, f\rangle)$, where $\LIM$ denotes a Banach limit. By the Riesz Representation Theorem $\nu(f)=\langle \nu, f\rangle$, where $\nu\in\mathcal M_1(S)$ is invariant.

An operator $P$ is called {\it asymptotically stable} if it has a unique invariant measure $\mu_*\in\mathcal{M}_1(S)$ such that the sequence $(P^n\mu)$ converges in the weak-$*$ topology to $\mu_*$ for any $\mu\in\mathcal{M}_1(S)$, i.e.,
$$
\lim_{n\to\infty} \langle P^n\mu, f\rangle=\langle \mu_*, f\rangle
\qquad\text{for any $f\in C(S)$.}
$$

In this paper we shall consider a special type of Feller operators. Assume that  $f_i : [0, 1]\to [0, 1]$ for $i=1,\ldots, N$ are continuous transformations and let $(p_1,\ldots, p_N)$ be a probability vector, i.e., $p_i\ge 0$ for all $i=1,..,N$ and $\sum_{i=1}^N p_i=1$.
The family $(f_1,..., f_N; p_1,...,p_N)$ generates a Markov operator $P : \mathcal{M}([0,1])\rightarrow \mathcal{M}([0,1])$ of the form
\begin{equation}\label{e1111_23.06.19}
P\mu(A) = \sum_{i=1}^N p_i\mu(f_i^{-1}(A))\quad\text{for $A\in \mathcal{B}([0,1])$.} 
\end{equation}
 This Markov operator is a Feller operator and its predual operator $U : C([0,1])\rightarrow C([0,1])$ is given by the formula
$$
U\varphi(x)=\sum_{i=1}^N p_i \varphi(f_i(x)) \ \textrm{for $\varphi\in C([0,1])$ and $x\in [0,1]$.} 
$$
By induction we check that 
\begin{equation}\label{e2_12.09.19}
U^n\varphi(x)=\sum_{i_1=1}^N\cdots\sum_{i_n=1}^Np_{i_1}\cdots p_{i_n}\varphi (f_{i_1}\circ\cdots\circ f_{i_n}(x))
\end{equation}
for $n\in\mathbb N$, $\varphi\in C([0,1])$ and $x\in [0,1]$. 
\vskip3mm
Markov operators corresponding to random transformations have been intensively studied for many years. In particular, W. Doeblin and R. Fortet in \cite{DF} considered the case when the maps $f_i$ were strict contractions but the probabilities $p_i$ were dependent on position, but Lipschitz functions.  
S.R. Foguel and B. Weiss in \cite{FW} considered convex combinations of commuting contractions in Banach spaces. R.~Sine in \cite{Sine} studied random rotations of the unit circle with position dependent probabilities $p_i$. In turn, the connections of random transformations to fractals have been discovered by J. Hutchinson in \cite{H}. M. Barnsley and R. Demko coined the term {\it iterated function systems} for systems with contractions (see \cite{BD}). In \cite{BDEG} the authors considered functions systems contractive on the average in the case where the state space $S$ is locally compact (see also \cite{LY}). Their result on asymptotic stability was extended to Polish spaces in \cite{SZ}. Random transformations, more general than iterated function systems have been also studied, but for more details we refer the reader to Kifer's book (see \cite{Kifer}).
\vskip3mm

We start with the following definitions.

\begin{df}
Let $H^+$ be the space of homeomorphisms $f : [0,1]\rightarrow [0,1]$ satisfying the following properties:
\begin{enumerate}
\item $f$ is increasing,
\item $f$ is a $C^1$--function in the neighbourhood of $0$ and $1$.
\end{enumerate}
\end{df}

\begin{df}
Let $\{f_1,\ldots, f_N\}\subseteq H^+$ be a finite collection of homeomorphisms and let $(p_1,\ldots,p_N)$ be a probability vector such that $p_i>0$ for all $i=1,\ldots,N$. The family $(f_1,\ldots, f_N; p_1,\ldots, p_N)$ is called an admissible iterated function system if
\begin{enumerate}
\item for any $x\in (0,1)$ there exist $i, j\in \{1,\ldots, N\}$ such that $f_i(x)<x<f_j(x)$,
\item $f_i'(0)>0$ and $f_i'(1)>0$ and the Lyapunov exponents at both points $0, 1$ are positive, i.e.,
$$\sum_{i=1}^Np_i\log f_i'(0)>0 \ \textrm{and} \ \sum_{i=1}^Np_i\log f_i'(1)>0.$$
\end{enumerate}
\end{df}

We set $\Sigma=\{1,...,N\}^\mathbb{N}$ and $\Sigma_n=\{1,\ldots, N\}^n$ for $n\in\mathbb N$. 
Put $\Sigma_*=\bigcup_{n=1}^\infty \Sigma_n$. Clearly, a probability vector $(p_1,\ldots,p_N)$ on $\{1,\ldots,N\}$ defines the product measures $\mathbb{P}, \mathbb{P}_n$ on $\Sigma$ and $\Sigma_n$ for $n\in\mathbb N$, respectively. The expected value with respect to $\mathbb P$ is denoted by $\mathbb E$. For any $n\in\mathbb N$ and $\textbf{i}=(i_1,i_2,\ldots)\in \Sigma$ we set $\textbf{i}_{|n}=(i_1,i_2,\ldots,i_n)$. In the same way we define $\textbf{i}_{|n}$ for $\textbf{i}=(i_1,\ldots,i_k)\in\Sigma_k$ with $k\geq n$. Additionally, we assume that $\textbf{i}_{|0}$ is the empty sequence for any $\textbf{i}\in\Sigma\cup\Sigma_*$. For a sequence ${\bf i}\in\Sigma_*$, ${\bf i}=(i_1,\dots ,i_n)$, we denote by $|{\bf i}|$ its length (equal to $n$). We shall write $f_{\textbf{i}}=f_{i_n}\circ f_{i_{n-1}}\circ\cdots\circ f_{i_1}$ for any sequence $\textbf{i}=(i_1,\ldots,i_n)\in\Sigma_n$, $n\in\mathbb N$.

 Let $\sigma : \Sigma\rightarrow \Sigma$ denote the shift transformation, i.e., $\sigma((i_1, i_2,\ldots.)) =(i_2, i_3\ldots)$ for $(i_1, i_2,\ldots)\in\Sigma$. If $\mathbf i=(i_1,\ldots,i_n)\in\Sigma_*$ and $\mathbf j= (j_1,\ldots,j_k)\in \Sigma_*$, then by $\mathbf i\mathbf j$ we denote the concatenation of $\mathbf i $ and $\mathbf j$, i.e., the sequence $(i_1,\ldots, i_n, j_1,\ldots, j_k)\in \Sigma_*$. If $\mathbf i\in \Sigma_*$ and $\mathbf j\in \Sigma$, then we can define concatenation $\mathbf i\mathbf j$ of sequences $\mathbf i$ and $\mathbf j$  in the same way, obtaining the sequence from the space $\Sigma$.  We write ${\mathbf i}\prec{\mathbf j}$ for $\mathbf i\in \Sigma_*, \mathbf j\in\Sigma\cup\Sigma_*$ if there exists ${\mathbf k}\in\Sigma\cup\Sigma_*$ such that ${\mathbf i}{\mathbf k}={\mathbf j}$.
\vskip2mm

Let an admissible iterated function system $(f_1,\ldots, f_N; p_1,\ldots,p_N)$ be given and let $P$ be the corresponding Markov operator defined by formula (\ref{e1111_23.06.19}). For every measure $\nu\in\mathcal M_1$ the law of the Markov chain $(X_n)$ with transition probability $\pi (x, A)=P\delta_x(A)$ for $x\in [0, 1]$, $A\in\mathcal B([0, 1])$ and initial distribution $\nu$, is the probability measure $\mathbb P_{\nu}$ on $([0, 1]^{\mathbb N}, \mathcal B([0, 1])^{\otimes\mathbb N})$ such that:
$$
\mathbb P_{\nu}[X_{n+1}\in A|X_n=x]=\pi (x, A)\quad\text{and}\quad\mathbb P_{\nu}[X_0\in A]=\nu(A),
$$
where $x\in [0, 1]$, $A \in \mathcal B([0, 1])$. The existence of $\mathbb P_{\nu}$ follows from the Kolmogorov Extension Theorem. The expectation with respect to $\mathbb P_{\nu}$ is denoted by $\mathbb E_{\nu}$. For $\nu=\delta_x$, the Dirac measure at $x\in [0, 1]$, we write just $\mathbb P_x$ and $\mathbb E_x$.
Obviously $\mathbb P_{\nu}(\cdot)=\int_{[0, 1]}\mathbb P_x(\cdot)\nu(\d x)$ and $\mathbb E_{\nu}(\cdot)=\int_{[0, 1]}\mathbb E_x(\cdot)\nu(\d x)$. Observe that for $n\in\mathbb N$ and $A_0,\ldots, A_n\in\mathcal B([0, 1])$ we have
$$
\begin{aligned}
&\mathbb P_x((X_0,\ldots, X_n)\in A_0\times\cdots\times A_n))\\
&=\sum_{(i_1,\ldots, i_n)\in\Sigma_n}({\bf 1}_{A_1\times\cdots\times A_n}(f_{i_1}(x),\ldots, f_{(i_n,\ldots, i_1)}(x))p_{i_1}\cdots p_{i_n})\\
&=\int_{\Sigma_n}{\bf 1}_{A_1\times\cdots\times A_n}(f_{i_1}(x),\ldots, f_{(i_n,\ldots, i_1)}(x))\mathbb {P}_n(\d {\bf i})\\
&=\int_{\Sigma}{\bf 1}_{A_1\times\cdots\times A_n}(f_{i_1}(x),\ldots, f_{(i_n,\ldots, i_1)}(x))\mathbb {P}(\d {\bf i}).
\end{aligned}
$$
Consequently,
\begin{equation}\label{e120_23.09.19}
\mathbb E_x(H(X_0,\ldots, X_n))=\int_{\Sigma}H(f_{i_1}(x),\ldots, f_{(i_n,\ldots, i_1)}(x))\mathbb {P}(\d {\bf i})
\end{equation}
and
\begin{equation}\label{e121_23.09.19}
\mathbb E_{\nu}(H(X_0,\ldots, X_n))=\int_{[0, 1]}\int_{\Sigma}H(f_{i_1}(x),\ldots, f_{(i_n,\ldots, i_1)}(x))\mathbb {P}(\d {\bf i})\nu(\d x)
\end{equation}
for an arbitrary bounded Borel--measurable function $H: [0, 1]^n\to\mathbb C$.

\vskip3mm
For $\alpha\in (0, 1)$ and $M\geq 1$  we define the sets $\mathcal{P}_{M,\alpha}^-, \mathcal{P}_{M,\alpha}^+, \mathcal{P}_{M,\alpha} $ as follows:
$$\mathcal{P}_{M,\alpha}^-:=\{\mu\in\mathcal{M}_1([0, 1]): \mu([0,x])\leq Mx^\alpha \  \textrm{for all}\\ \ x\in[0,1]\\ \},$$
$$\mathcal{P}_{M,\alpha}^+:=\{\mu\in\mathcal{M}_1([0, 1]): \ \mu([1-x,1])\leq Mx^\alpha\\ \ \textrm{for all}\\ \ x\in[0,1]\\ \},$$
$$\mathcal{P}_{M, \alpha}:=\mathcal{P}_{M,\alpha}^-\cap \mathcal{P}_{M,\alpha}^+.$$

For $\varepsilon>0$ and  $x<\varepsilon$ we set
$$
A_{x, k}( \varepsilon ):=\{{\bf i}\in\Sigma : f_{\bf i}^k(x)<\varepsilon\}\quad\text{for $k\in\mathbb N$},
$$
where $f_{\bf i}^k=f_{{\bf i}|k}$.
Similarly, for $\varepsilon>0$ and  $x>1-\varepsilon$ we set
$$
A^{x, k}( \varepsilon ):=\{{\bf i}\in\Sigma: f_{\bf i}^k(x)>1-\varepsilon\}\quad\text{for $k\in\mathbb N$}.
$$


\section{Invariant measures}
Let an admissible iterated function system $(f_1,\ldots, f_N; p_1,\ldots, p_N)$ be given and let $P$ be the corresponding Markov operator. The Markov operator $P$ clearly admits two trivial invariant measures: $\delta_0$ and $\delta_1$. In this section we prove the existence of a unique invariant measure $\mu\in\mathcal M_1((0, 1))$. Let us stress that this fact was already shown in \cite{Homburg} and our proof of existence of an invariant measure is just a repetition of the argument. However, the proof of uniqueness is new, more general and elementary. Similarly the proof of stability on the interval $(0, 1)$.  We start with the following lemma.


\begin{lem}
\label{constants}
Let $(f_1,\ldots, f_N; p_1,\ldots, p_N)$ be an admissible iterated function system and let $P$ be the corresponding Markov operator. Then there exist constants $\varepsilon\in (0, \frac 1 2)$, $\alpha, \delta\in (0, 1)$ and $M\ge 1$ such that $P(\mathcal{P}^{\pm}_{M,\alpha})\subseteq \mathcal{P}^{\pm}_{M,\alpha}$ and
\begin{equation}
\label{delta_in_class}
\delta_x\in\mathcal{P}^-_{M, \alpha} \  \text{for} \ x\geq \varepsilon \ \text{and} \ \delta_x\in\mathcal{P}^+_{M, \alpha} \ \text{for} \ x\leq 1-\varepsilon.
\end{equation}
Moreover, for any $n\in\mathbb N$ we have
\begin{equation}
\label{probability_A_0}
\mathbb{P}\bigg(\bigcap_{j=1}^{\lfloor\sqrt[4]{n}\rfloor}A_{x, j}(\varepsilon)\bigg)\leq (1-\delta)^{\frac{\alpha}{2}\lfloor\sqrt[4]{n}\rfloor}\quad\text{for} \ x\in[\varepsilon(1-\delta)^{\frac 1 2 \lfloor\sqrt[4]{n}\rfloor}, \varepsilon]
\end{equation}
and
\begin{equation}
\label{probability_A_1}
\mathbb{P}\bigg(\bigcap_{j=1}^{\lfloor\sqrt[4]{n}\rfloor}A^{x, j}(\varepsilon)\bigg)\leq (1-\delta)^{\frac{\alpha}{2}\lfloor\sqrt[4]{n}\rfloor}\quad\text{for} \ x\in [1-\varepsilon, 1-\varepsilon(1-\delta)^{\frac 1 2 \lfloor\sqrt[4]{n}\rfloor}].
\end{equation}
\end{lem}

\begin{proof}
By the definition of an admissible iterated function system we can find positive numbers $\underline{\lambda}_1<f_1'(0),\ldots, \underline{\lambda}_N<f_N'(0), \overline{\lambda}_1<f'_1(1),\ldots, \overline{\lambda}_N<f'_N(1)$ such that still
\begin{equation}
\label{lyapunov_exponents}
\sum_{i=1}^N p_i\log \underline{\lambda}_i>0\quad\text{and}\quad \sum_{i=1}^N p_i\log \overline{\lambda}_i>0.
\end{equation}
Using the definition of derivative at points $0$ and $1$ we can choose $\varepsilon<\frac 1 2$ such that
\begin{equation}
\label{lambdas_estimation}
f_i(x)\geq \underline{\lambda}_ix\quad\text{and} \quad f_i(1-x)\leq \overline{\lambda}_i(1-x)
\end{equation}
for $i=1,\ldots, N$ and $x\leq \varepsilon$. Now we use the definition of derivative for functions $f^{-1}_1,\ldots, f^{-1}_N$ at points $0$ and $1$ and correct the choice of $\varepsilon$ to satisfy the conditions:
\begin{equation}
\label{lambdas_estimation_inv}
f^{-1}_i(x)\leq \frac{x}{\underline{\lambda}_i}\quad\text{and}\quad f^{-1}_i(1-x)\geq 1-\frac{x}{\overline{\lambda}_i}
\end{equation}
for $i=1,\ldots, N$ and $x\leq \varepsilon$.

Writing the Taylor expansion of the function $\lambda^{-\alpha}$ at 0 with respect to $\alpha$ we obtain $\lambda^{-\alpha}=1-\alpha\log\lambda +o(\alpha)$. Therefore, using (\ref{lyapunov_exponents}) we can fix $\alpha\in (0,1)$ and $\delta>0$ such that we have
\begin{equation}
\label{expectation_lambdas}
\sum_{i=1}^N p_i\underline{\lambda_i}^{-\alpha}<(1-\delta)^\alpha\quad\text{and} \quad \sum_{i=1}^N p_i\overline{\lambda_i}^{-\alpha}<(1-\delta)^\alpha.
\end{equation}
Finally, put $M=\varepsilon^{-\alpha}$. This immediately implies condition (\ref{delta_in_class}).

Now we are showing the invariance of $\mathcal{P}_{M, \alpha}$. Let $\mu \in \mathcal{P}^-_{M, \alpha}$ and $x\in (0,1)$. We are going to show that $P\mu\in \mathcal{P}^-_{M, \alpha}$. If $x\geq \varepsilon$, then $Mx^\alpha\geq M\varepsilon^\alpha= 1$, hence obviously $P\mu([0,x])\leq Mx^\alpha$. If $x<\varepsilon$, then
$$
\begin{aligned}
P\mu([0,x])&=\sum_{i=1}^N p_i \mu([0, f_i^{-1}(x)])\leq \sum_{i=1}^N p_i \mu\bigg(\bigg[0, \frac{x}{\underline{\lambda}_i}\bigg]\bigg)
\leq \sum_{i=1}^N p_i M\bigg(\frac{x}{\underline{\lambda}_i}\bigg)^\alpha\\
&=Mx^\alpha \sum_{i=1}^N p_i \underline{\lambda}_i^{-\alpha}<Mx^\alpha(1-\delta)^\alpha\leq Mx^\alpha.
\end{aligned}
$$
Analogous computations for $P^+_{M, \alpha}$ complete this part of the proof.

We set ${\lambda}_{\bf i}^k:={\lambda}_{i_k}{\lambda}_{i_{k-1}}\cdots {\lambda}_{i_1}$ for arbitrary real numbers $\lambda_i$, $1\le i\le N$, and ${\bf i}=(i_1, i_2,\ldots)\in\Sigma$. Applying (\ref{expectation_lambdas}) we have
$$
\mathbb{E}\big( (\underline{\lambda}^{\lfloor \sqrt[4]{n} \rfloor}_{\bf i})^{-\alpha} \big)\leq (1-\delta)^{\alpha\lfloor \sqrt[4]{n} \rfloor}\quad\text{for any $n\in\mathbb{N}$.}
$$
By Chebyshev's inequality and (\ref{lambdas_estimation}) for $x\in[\varepsilon(1-\delta)^{\frac 1 2 \lfloor\sqrt[4]{n}\rfloor}, \varepsilon]$ we have
\begin{align*}
&\mathbb{P}\bigg(\bigcap_{j=1}^{\lfloor\sqrt[4]{n}\rfloor}A_{x, j}(\varepsilon)\bigg)=\mathbb{P}\bigg(\{{\bf i}\in\Sigma : \forall_{j\leq \lfloor\sqrt[4]{n}\rfloor} \ f_{\bf i}^j(x)\leq\varepsilon   \}\bigg)\\
&\leq\mathbb{P}\bigg(\{{\bf i}\in\Sigma : \forall_{j\leq \lfloor\sqrt[4]{n}\rfloor} \ f_{\bf i}^j \big(\varepsilon(1-\delta)^{\frac{1}{2}\lfloor \sqrt[4]{n} \rfloor}\big)\leq\varepsilon   \}\bigg)\\
&\leq\mathbb{P}\bigg(\{{\bf i}\in\Sigma : \forall_{j\leq \lfloor\sqrt[4]{n}\rfloor} \  \underline{\lambda}^{j}_{\bf i}\varepsilon(1-\delta)^{\frac{1}{2}\lfloor \sqrt[4]{n} \rfloor}\leq\varepsilon   \}\bigg)\\
&\leq\mathbb{P}\bigg(\{{\bf i}\in\Sigma : (1-\delta)^{\frac{1}{2}\lfloor\sqrt[4]{n}\rfloor}<\underline{\lambda}_{\bf i}^{-\lfloor\sqrt[4]{n}\rfloor}\}\bigg)
\leq (1-\delta)^{-\frac{\alpha}{2}\lfloor\sqrt[4]{n}\rfloor}\mathbb{E}\big(\underline{\lambda}_{\bf i}^{-\alpha\lfloor\sqrt[4]{n}\rfloor}\big)\\
&\leq (1-\delta)^{-\frac{\alpha}{2}\lfloor\sqrt[4]{n}\rfloor}(1-\delta)^{\alpha \lfloor\sqrt[4]{n}\rfloor}
\leq (1-\delta)^{\frac{\alpha}{2}\lfloor\sqrt[4]{n}\rfloor}\qquad\text{for $n\in\mathbb N$}.
\end{align*}
Estimate (\ref{probability_A_1}) may be proved in the same way.
\end{proof}


\begin{thr}\label{T1_23.06.19}
If $(f_1,\ldots, f_N; p_1,\ldots, p_N)$ is an admissible iterated function system, then the corresponding Markov operator $P$ has a unique invariant measure $\mu_*\in\mathcal M_1((0, 1))$. Moreover $\mu_*$ is atomless.
\end{thr}
\begin{proof} {\bf Existence.} Let $\varepsilon, \alpha, \delta, M$ be the positive constants given in Lemma \ref{constants}. Then $\delta_{\frac{1}{2}}\in\mathcal{P}_{M, \alpha}$ and, by Lemma \ref{constants},

$$\mu_n:=\frac{1}{n}(\delta_{\frac{1}{2}}+P\delta_{\frac{1}{2}} +\cdots+P^{n-1}\delta_{\frac{1}{2}})\in \mathcal{P}_{M, \alpha} \ \textrm{for $n\in\mathbb N$.} $$
Since the family of measures $\mathcal{P}_{M, \alpha}$ is tight, there exists an accumulation point $\mu\in\mathcal M_1([0, 1])$ of the sequence $(\mu_n)$ in the weak-$*$ topology. It is easy to check that $\mu\in\mathcal{P}_{M, \alpha}$, which obviously implies that $\mu(\{0,1\})=0$. Moreover, the operator $P$ is a Feller operator and henceforth the measure $\mu$ is invariant for $P$. This completes the proof of existence.

{\bf Uniqueness.} Since for any $x\in (0,1)$ there exist $i, j\in \{1,\ldots, N\}$ such that $f_i(x)<x<f_j(x)$, $0$ and $1$ belong to the support of every invariant measure $\mu\in\mathcal M_1((0, 1))$.

Assume, contrary to our claim, that there exist at least two different ergodic invariant measures $\mu_1, \mu_2\in\mathcal M_1((0, 1))$. Then we may choose $\xi\in (0,1)$ such that, say, $\mu_2((0,\xi))>\mu_1((0,\xi))$. Since $\mu_1$ is ergodic, by the Birkhoff Ergodic Theorem we may choose  $x_1\in(0,1)$ such that 
$$
\lim_{n\to\infty} \frac{\#\{1\leq k\leq n: f^k_{\bf i}(x_1)\in (0,\xi)\}}{n}=\mu_1\big((0, \xi)\big) \ \textrm{for $\mathbb{P}$-a.e. ${\bf i}\in\Sigma$}.
$$
On the other hand, since $\mu_2((x_1, 1))>0$ ($1$ belongs to the support of $\mu_2$), we can also choose $x_2\in (x_1, 1)$ such that
$$
\lim_{n\to\infty} \frac{\#\{1\leq k\leq n: f^k_{\bf i}(x_2)\in (0,\xi)\}}{n}=\mu_2\big((0,\xi)\big) \ \textrm{for $\mathbb{P}$-a.e. ${\bf i}\in\Sigma$},
$$
again by the Birkhoff Ergodic Theorem. But the functions $f_1,\ldots, f_N$ are increasing, $x_1<x_2$ and therefore
$$ \frac{\#\{1\leq k\leq n: f^k_{\bf i}(x_1)\in (0,\xi)\}}{n}\geq \frac{\#\{1\leq k\leq n: f^k_{\bf i}(x_2)\in (0,\xi)\}}{n}\quad\text{for all ${\bf i}\in\Sigma$.}
$$
From this we conclude that $\mu_1((0,\xi))\geq \mu_2((0,\xi))$, which contradicts our assumption. 

Finally assume, contrary to our claim, that $\mu_*$ has an atom. Let $a\in (0, 1)$ be such that $\mu_*(\{a\})=\sup_{x\in (0, 1)}\mu_*(\{x\})>0$. Since
$$
\mu_*(\{a\})=\sum_{i=1}^N p_i\mu_*(\{f_i^{-1}(a)\}),
$$
we obtain that $\mu_*(\{f_i^{-1}(a)\})=\mu_*(\{a\})$ for all $i=1,\ldots, N$ and consequently $\mu_*(\{a\})=\mu_*(\{f_{\bf i}^{-1}(a)\})$ for any ${\bf i}\in\Sigma_*$, which is impossible. Indeed, the set $\{f_{\bf i}^{-1}(a): {\bf i}\in\Sigma_*\}$ for admissible iterated function systems  is infinite. This, in turn,  would imply that $\mu_*((0, 1))=\infty$. The proof is complete.
\end{proof}

\begin{thr}\label{T2_23.06.19}
Let $(f_1,\ldots f_N; p_1,\ldots,p_N)$ be an admissible iterated function system and let  $P$ be the corresponding Markov operator. Let $\mu_*\in\mathcal M_1((0, 1))$ be its unique invariant measure. Then for any measure $\mu\in\mathcal M_1((0, 1))$ we have
$$
\lim_{n\to\infty}\langle P^n\mu, \varphi\rangle=\langle \mu_*, \varphi\rangle\qquad\text{for $\varphi\in C([0, 1])$.}
$$
\end{thr}

\begin{proof} We follow \cite{Furstenberg} in defining some martingale. Namely, for $\varphi\in C([0, 1])$ we consider the sequence of random variables   
$(\xi_n^{\varphi})$ defined on the probability space $(\Sigma, \mathbb P)$ by the formula
$$
\xi_n^{\varphi}({\bf i})=\langle \mu_*, \varphi \circ f_{(i_1,\ldots, i_n)}\rangle \quad\text{for ${\bf i}=(i_1, i_2,\ldots)$}.
$$
Since $\mu_*$ is an invariant measure for $P$, we easily check that $(\xi_n^{\varphi})$ is a bounded martingale and from the Martingale Convergence Theorem it follows that $(\xi_n^{\varphi})$ is convergent $\mathbb P$-a.s. Since the space $C([0, 1])$ is separable, there exists a subset $\Sigma_0\subset\Sigma$  with $\mathbb P(\Sigma_0)=1$ such that $(\xi_n^{\varphi}({\bf i}))$ is convergent for any $\varphi\in C([0, 1])$ and ${\bf i}\in\Sigma_0$. By the Riesz Representation Theorem for any ${\bf i}\in\Sigma_0$ there exists a measure $\mu_{\bf i}\in\mathcal M_1([0, 1])$ such that
\begin{equation}\label{e1_12.08.19}
\lim_{n\to\infty}\xi_n^{\varphi}({\bf i})= \langle \mu_{\bf i}, \varphi\rangle\qquad\text{for every $\varphi\in C([0, 1]$}.
\end{equation}

Now we are going to show that $\mu_{\bf i}$ is supported at some point $\upsilon({\bf i})\in [0, 1]$ for $\mathbb P$-a.e. ${\bf i}\in\Sigma$. To do this it is enough to show that for any $\varepsilon>0$ there exists $\Sigma_{\varepsilon}\subset\Sigma_0$ with $\mathbb P(\Sigma_{\varepsilon})=1$ satisfying the following property:  for every ${\bf i}\in\Sigma_{\varepsilon}$ there exists an interval $I$ of length $|I|\le\varepsilon$ such that $\mu_{\bf i}(I)\ge 1-\varepsilon$. Hence we obtain that $\mu_{\bf i}=\delta_{\upsilon({\bf i})}$ for all ${\bf i}$ from the set $\tilde\Sigma_0=\bigcap_{n=1}^{\infty}\Sigma_{1/n}$. Obviously $\mathbb P(\tilde\Sigma_0)=1$.

Fix $\varepsilon>0$ and let $a, b\in (0, 1)$ be such that $\mu_* ([a, b])>1-\varepsilon$. Let $l\in\mathbb N$ be such that $1/l<\varepsilon/2$. Since for any $x\in (0, 1)$ there exists $i\in\{1,\ldots, k\}$ such that $f_i(x)<x$, we may find a sequence $({\bf j}_n)$, ${\bf j}_n\in\Sigma_*$, such that $f_{{\bf j}_n}(b)\to 0$ as $n\to\infty$. Therefore, there exist ${\bf i}_1,\ldots, {\bf i}_l$ such that $f_{{\bf i}_m}([a, b])\cap f_{{\bf i}_n}([a, b])=\emptyset$ for $m, n\in\{1,\ldots, l\}$, $m\neq n$.
Put $n^*=\max_{m\le l}|{\bf i}_m|$ and set $J_m=f_{{\bf i}_m}([a, b])$ for $m\in\{1,\ldots, l\}$.
Now observe that for any sequence ${\bf u}=(u_1,\ldots, u_n)\in \Sigma_*$ there exists $m\in\{1,\ldots, l\}$ such that $|f_{\bf u}(J_m)|<1/l<\varepsilon/2$. 
This shows that for any cylinder  in $\Sigma$, defined by fixing the first initial $n$ entries $(u_1,\ldots, u_{n})$, the conditional probability that  $(u_1,\ldots, u_n,\ldots, u_{n+k})$ are such that
$|f_{(u_1, \ldots, {u_n},\ldots, u_{n+k})}([a, b])|\ge\varepsilon$ for all $k=1,\ldots, n^*$ is less than $1-q$ for some $q>0$. Hence there exists $\Sigma_{\varepsilon}\subset\Sigma$ with $\mathbb P(\Sigma_{\varepsilon})=1$ such that for all $(u_1, u_2,\ldots)\in\Sigma_{\varepsilon}$ we have $|f_{(u_1,\ldots, u_n)}([a, b])|<\varepsilon/2$ for infinitely many $n$. Since $[0, 1]$ is compact, we may additionally assume that for infinitely many $n$'s the set $f_{(u_1,\ldots, u_n)}([a, b])$ is contained in some set $I$ with $|I|\le\varepsilon$.
Since $\mu_*$ is an invariant probability measure and $a, b\in (0, 1)$ are chosen in such a way that $\mu_*([a, b])>1-\varepsilon$, we obtain that $\mu_{\bf i}(I)\ge 1-\varepsilon$. 
The proof of our assertion that $\mu_{\bf i}$ is supported at some point $\upsilon({\bf i})\in [0, 1]$ for $\mathbb P$-a.e. ${\bf i}\in\Sigma$ is finished. 

To show that the sequence $(P^n\mu)$ for $\mu\in\mathcal M_1((0, 1))$ converges weakly to $\mu_*$ it is enough to prove, since the Lipschitz functions are dense in $C([0, 1])$, that for any Lipschitz function $\varphi$ and arbitrary two points $x, y\in (0, 1)$ we have
$$
\lim_{n\to\infty}\left|\langle P^n\delta_x, \varphi\rangle-\langle P^n\delta_y, \varphi\rangle\right|=0.
$$
In fact, we would obtain then that 
$$
\begin{aligned}
&\left|\langle P^n\mu, \varphi\rangle- |\langle \mu_*, \varphi\rangle\right|\\
&\le\int_{(0, 1)}\int_{(0, 1)}\left|\langle P^n\delta_x, \varphi\rangle-
\langle P^n\delta_y, \varphi\rangle\right|\mu(\d x)\mu_*(\d y)\to 0\quad\text{as $n\to\infty$}.
\end{aligned}
$$

Fix $x, y\in (0, 1)$ and let $x<y$. Fix $\varepsilon>0$. Since $\mu_*$ is invariant, by the proof of uniqueness in Theorem \ref{T1_23.06.19}, $0$ and $1$ belong to its support and consequently 
$\mu_* ((0, x))>0$ and $\mu_* ((y, 1))>0$. We know by (\ref{e1_12.08.19})
that for $\mathbb P$ almost every ${\bf i}=(i_1, i_2,\ldots)\in\Sigma$ the measures $\mu_*\circ f_{(i_1, i_2,\ldots, i_n)}^{-1}$, $n\in\mathbb N$, converge weakly to $\delta_{\upsilon({\bf i})}$. Consequently, for every $\varepsilon>0$, $\mu_* \circ f_{(i_1, i_2,\ldots, i_n)}^{-1}((\upsilon({\bf i})-\varepsilon/2, \upsilon({\bf i})+\varepsilon/2)\cap (0, 1))\to 1$ as $n\to\infty$. Since $\mu_*((0, x))>0$ and $\mu_*((y, 1))>0$, there exist $u_n\in (0, x)$ and $v_n\in (y, 1)$ such that $u_n, v_n\in f_{(i_1,\ldots, i_n)}^{-1}((\upsilon({\bf i})-\varepsilon/2, \upsilon({\bf i})+\varepsilon/2)\cap (0, 1))$ for all $n$ sufficiently large. Hence $f_{(i_1, \ldots, i_n)}(u_n),  f_{(i_1,\ldots, i_n)}(v_n)\in (\upsilon({\bf i})-\varepsilon/2, \upsilon({\bf i})+\varepsilon/2)$ and consequently $f_{(i_1,\ldots, i_n)}(x),  f_{(i_1,\ldots, i_n)}(y)\in (\upsilon({\bf i})-\varepsilon/2, \upsilon({\bf i})+\varepsilon/2)$ for all $n$ sufficiently large. Since $\varepsilon>0$ was arbitrary, we obtain that for $\mathbb P$--a.e. ${\bf i}=(i_1, i_2,\ldots)\in\Sigma$ the following convergence holds:
$$
\lim_{n\to\infty} |f_{(i_1,\ldots, i_n)}(x)-f_{(i_1,\ldots, i_n)}(y)|=0.
$$
By (\ref{e2_12.09.19}) and the fact that $\langle P^n\delta_z, \varphi\rangle=U^n\varphi(z)$ for $z\in [0, 1]$ we have
\begin{equation}\label{e3_12.08.19}
\begin{aligned}
&\left|\langle P^n\delta_x, \varphi\rangle-
\langle P^n\delta_y, \varphi\rangle\right|=|U^n\varphi(x)-U^n\varphi(y)|\\
&\le L\sum_{(i_1,\ldots, i_n)\in\Sigma_n}|f_{(i_1,\ldots, i_n)}(x)-f_{(i_1,\ldots, i_n)}(y)|p_{i_1}\cdots p_{i_n}\quad\text{for $x, y\in (0, 1)$},
\end{aligned}
\end{equation}
where $L$ is the Lipschitz constant of $\varphi$. We are going to show that the right hand side of the above inequality converges to $0$ as $n\to\infty$. To do this for ${\mathbf i}=(i_1,\ldots, i_n,\ldots)$ put $g_n({\mathbf i}):=
|f_{\mathbf i_n}(x)-f_{\mathbf i_n}(y)|$, where ${\mathbf i_n}=(i_1,\ldots, i_n)$. Then $g_n({\mathbf i})\to 0$ as $n\to\infty$ for $\mathbb P$ almost every ${\mathbf i}\in\Sigma$. By the construction of the probability measures $\mathbb P$ and $\mathbb P_n$, $n\in\mathbb N$, we have
$$
\mathbb P(B_n\times\Sigma_1\times\Sigma_1\times\cdots)=\mathbb P_n(B_n)\quad\text{for $B_n\in\Sigma_n$}.
$$
Since $g_n({\mathbf i})$ depends only on the first $n$ coordinates, we have
$$
\int_{\Sigma}g_n({\mathbf i})\d \mathbb P({\mathbf i})=\int_{\Sigma_n}g_n({\mathbf i})\d \mathbb P_n({\mathbf i})=\sum_{(i_1,\ldots, i_n)\in\Sigma_n}|f_{(i_1,\ldots, i_n)}(x)-f_{(i_1,\ldots, i_n)}(y)|p_{i_1}\cdots p_{i_n}.
$$
Therefore, by the Lebesgue theorem we obtain
$$
\lim_{n\to\infty}\sum_{(i_1,\ldots, i_n)\in\Sigma_n}|f_{(i_1,\ldots, i_n)}(x)-f_{(i_1,\ldots, i_n)}(y)|p_{i_1}\cdots p_{i_n}=0
$$
and inequality (\ref{e3_12.08.19}) completes the proof.
\end{proof}

\begin{rem} From uniqueness in Theorem \ref{T1_23.06.19} and Theorem 1.3 in \cite{Krengel} it follows that for any $f\in C([0, 1])$ satisfying $f(0)=f(1)=0$ and $\langle \mu_*, f\rangle=0$ we have
$$
\left\|\frac{1}{n}\sum_{k=1}^n U^k f\right\|\to 0\qquad\text{as $n\to\infty$.}
$$
On the other hand, the asymptotic stability in Theorem \ref{T2_23.06.19} raises the question whether we have then $\|U^n f\|\to 0$ as $n\to\infty$. Unfortunately, we have not been able to answer the question, neither affirmatively nor negatively. 
\end{rem}

\begin{cor} Under the assumptions of Theorem \ref{T2_23.06.19} we have:
\begin{itemize}
\item[(i)] for every $f\in C([0, 1])$ we have the convergence of $(U^nf(x))$ for every $x\in [0, 1]$,
\item[(ii)] for every $f\in L_2(\mu_*)$ we have $\|U^nf-\int_{[0, 1]}f\d\mu_*\|_2\to 0$ as $n\to\infty$.
\end{itemize}
\end{cor}

\begin{proof} (i) For  $x\in (0, 1)$ we have $U^n f(x)\to\int_{[0, 1]}f\d\mu_*$ by applying Theorem \ref{T2_23.06.19} to $\delta_x$. Since $U^n f(0)=f(0)$ and $U^n f(1)=f(1)$, we have the required convergence (not uniform since the limit may not be continuous).

(ii) Since $\mu_*(\{0, 1\})=0$, we have $\|U^nf-\int_{[0, 1]}f\d\mu_*\|_2\to 0$ as $n\to\infty$ for every $f\in C([0, 1])$, by (i). Since $U$ is a contraction of 
$L_2(\mu_*)$, this implies (ii).
\end{proof}

\section{Auxiliary results} For abbreviation we set, for $\varepsilon$ and $\delta$ of Lemma 1,
$$\varepsilon_n:=(1-\delta)^{\frac{1}{2}\lfloor\sqrt[4]{n}\rfloor}\quad\text{and}\quad
\gamma_n:=(1-\delta)^{\frac{\alpha}{2}\lfloor\sqrt[4]{n}\rfloor}\quad\text{for $n\in\mathbb N$.}
$$


\begin{lem}
\label{mass_boundary}
Let $(f_1,\ldots, f_N; p_1,\ldots, p_N)$ be an admissible iterated function system and let $\alpha, \delta, \varepsilon, M$ be the positive constants given in Lemma \ref{constants}. Then for all positive integers $n$ and $k\geq  \lfloor\sqrt[4]{n}\rfloor$ we have
$$
\mathbb{P}\big(\{{\bf i}\in\Sigma : f^{k}_{\bf i}(x)<\varepsilon_n\}\big)\leq 2M\gamma_n\quad\text{for $x\in [\varepsilon_n, 1]$}
$$
and
$$
\mathbb{P}\big(\{{\bf i}\in\Sigma : f^{k}_{\bf i}(y)>1-\varepsilon_n\}\big)\leq 2M\gamma_n\quad\text{for $y\in[0,1-\varepsilon_n]$.}
$$
\end{lem}

\begin{proof}
Let $\alpha, \delta, \varepsilon, M$ be the constants given in Lemma \ref{constants}. Fix $n\in\mathbb N$ and $x\in  [\varepsilon_n, 1] $, and  let $k\geq  \lfloor\sqrt[4]{n}\rfloor$. If $x\geq\varepsilon$, then we have 
\begin{equation}
\label{inequality_mass_boundary}
\begin{aligned}
\mathbb{P}\big(\{{\bf i}\in\Sigma : f^{k}_{\bf i}(x)<\varepsilon_n\}\big)=
P^k\delta_x([0,\varepsilon_n))
\leq M\varepsilon^\alpha \gamma_n<2M\gamma_n,
\end{aligned}
\end{equation}
where the last inequality follows from the fact that $\delta_x\in \mathcal{P}_{M,\alpha}^-$ and $P(\mathcal{P}_{M,\alpha}^-)\subseteq \mathcal{P}_{M,\alpha}^-$, by Lemma \ref{constants}.

Assume now that $x\in [\varepsilon_n,\varepsilon)$. Set
$$
E:=\{(i_1,\ldots, i_m)\in\Sigma_*: m\leq \lfloor\sqrt[4]{n}\rfloor, \forall_{r< m} f_{(i_1,\ldots, i_r)}(x)< \varepsilon \ \textrm{and} \ f_{(i_1,\ldots, i_m)}(x)\geq \varepsilon\}$$
and
$$
F:=\{(i_1,\ldots, i_k)\in \Sigma_k:  \forall_{m\leq \lfloor\sqrt[4]{n}\rfloor} f_{(i_1,\ldots, i_m)}(x)< \varepsilon\}.
$$
Then we have
$$
P^k\delta_x=\sum_{(i_1,\ldots, i_m)\in E} p_{i_1}\cdots p_{i_m}P^{k-m}\delta_{f_{(i_1,\ldots, i_m)}(x)}+\sum_{(i_1,\ldots, i_k)\in F} p_{i_1}\cdots p_{i_k}\delta_{f_{(i_1,\ldots, i_k)}(x)}.$$
By the definition of $E$ we see that $\delta_{f_{(i_1,\ldots, i_m)}(x)}\in \mathcal{P}_{M,\alpha}^-$ for $(i_1,\ldots, i_m)\in E$ and from this also $P^{k-m}\delta_{f_{(i_1,\ldots, i_s)}(x)}\in \mathcal{P}_{M,\alpha}^-$ for $(i_1,\ldots, i_m)\in E$. On the other hand, from Lemma \ref{constants} it follows that
$$
\sum_{(i_1,\ldots, i_k)\in F} p_{i_1}\cdots p_{i_k}\leq \mathbb{P}\bigg(\bigcap_{j=1}^{\lfloor\sqrt[4]{n}\rfloor}A_{x, j}(\varepsilon)\bigg)\leq \gamma_n.
$$
Consequently, we have
$$
\begin{aligned}
&\mathbb{P}\big(\{{\bf i}\in\Sigma : f^{k}_{\bf i}(x)<\varepsilon_n\}\big)\
\leq\sum_{(i_1,\ldots, i_m)\in E} p_{i_1}\cdots p_{i_m}P^{k-m}\delta_{f_{(i_1,\ldots, i_m)}(x)}\big([0, \varepsilon_n]\big)+\gamma_n\\
&\leq M\varepsilon_n^\alpha\gamma_n+\gamma_n\leq 2M\gamma_n\quad\text{for $x\in [\varepsilon_n,\varepsilon)$.}
\end{aligned}
$$
This and condition (\ref{inequality_mass_boundary}) completes the estimate for $x\in [\varepsilon_n, 1]$. The estimate for 
$y\in[0,1-\varepsilon_n]$ is proved in the same way. The proof is complete.
\end{proof}


Observe that if $(f_1,\ldots, f_N; p_1,\ldots, p_N)$ is admissible, then there are no nontrivial closed subintervals of $(0,1)$ invariant by $\{f_1,\ldots, f_N\}$. We are now in a position to recall the contracting result obtained by D. Malicet in \cite{Malicet}. This theorem is crucial in the proof and in our setting may be stated as follows.

\begin{thr}[cf. Corollary 2.13 in \cite{Malicet}]
\label{malicet}
If $(f_1,\ldots, f_N; p_1,\ldots, p_N)$ is an admissible iterated function system, then there exists $q\in (0,1)$ such that for every $x\in (0,1)$ and $\mathbb{P}$-a.e. ${\bf i}\in\Sigma$ we have a neighbourhood $I$ of $x$ in $(0, 1)$ satisfying
\begin{equation}
\label{malicet2}
|f^n_{\bf i}(I)|\leq q^n\quad\text{ for every $n\in\mathbb{N}$.}
\end{equation}
\end{thr}


\begin{lem}
\label{omega}
Let $(f_1,\ldots, f_N; p_1,\ldots, p_N)$ be an admissible iterated function system and let $a\in (0, \frac{1}{2})$. Then there exists $r\in\mathbb N$ and $\Omega\subseteq\Sigma$ with $\mathbb{P}(\Omega)>0$ such that for $J=[a,1-a]$ we have
$$
\sum_{n=1}^\infty \big| f_{\bf i}^n(J)\big| \leq r+\frac{q}{1-q}\qquad\text{for ${\bf i}\in\Omega$},
$$
where $q$ is the constant given in Theorem \ref{malicet}.
\end{lem}
\begin{proof} Fix $a\in (0, 1/2)$ and choose $x\in \textrm{supp} \  \mu_*\cap (0, 1)$. From Theorem \ref{malicet} we have $q\in (0, 1)$, a neighbourhood $I$ of $x$ and  $\widetilde{\Omega}\subseteq\Sigma$ with $\mathbb P ( \widetilde{\Omega} )>0$ such that
$$
|f^n_{\bf i}(I)|\leq q^n\qquad\text{for ${\bf i}\in\widetilde{\Omega}$ and $n\in\mathbb N$.}
$$
Set $b=\inf I$. We claim that there exists $r\in\mathbb N$ and a sequence $(i_1,\ldots, i_r)\in \Sigma_r $ such that $f_{(i_1,\ldots, i_r)}([a,1-a])\subseteq I$. Suppose, contrary to our claim, that for every $r\in\mathbb N$ and $(i_1,\ldots, i_r)\in \Sigma_r$, if $f_{(i_1,\ldots, i_r)}(1-a)\in I$, then $f_{(i_1, \ldots, i_r)}(a)\leq b$. This implies that
$$
P^r\delta_{a}(([0,b])\ge P^r\delta_{1-a}([0,b]\cup I)=P^r\delta_{1-a}([0,b])+P^r\delta_{1-a}(I)
$$
for every $r\in\mathbb N$. Since the iterated function system $(f_1,\ldots, f_N; p_1,\ldots, p_N)$ is asymptotically stable and $\mu_*$ is atomless, we have $P^r\delta_{a}([0,b])\to \mu_*([0, b])$, $P^r\delta_{1-a}([0,b])\to \mu_*([0,b])$,  $P^r\delta_{1-a}(I)\to\mu_*(I)$ as $r\to \infty$,  by Alexandrov's theorem. This is a contradiction with the above inequality, for $\mu_*(I)>0$.

Now choose $(i_1,\ldots, i_r)\in \Sigma_r $ such that $f_{(i_1,\ldots, i_r)}([a,1-a])\subseteq I$ and define
$$
\Omega= \{ (i_1,\ldots, i_r){\bf i}: {\bf i}\in \widetilde{\Omega} \}\in \Sigma.
$$
We see at once that $\Omega$ satisfies the desired condition. The proof is complete.
\end{proof}


\begin{lem}
\label{return_to_J}
Let $(f_1,\ldots, f_N; p_1,\ldots, p_N)$ be an admissible iterated function system and let $\alpha, \delta, \varepsilon, M$ be the positive constants given in Lemma \ref{constants}. Assume that $J=[a,1-a]$ for some $a\in (0,\frac{1}{2})$ such that $M=\varepsilon^{-\alpha}<\frac{a^{-\alpha}}{6}$. Then there exist $n_0\in\mathbb{N}$ such that for every integer $n\geq n_0$
$$
\mathbb{P}\big(\{{\bf i}\in\Sigma: f^{\lfloor\sqrt[4]{n}\rfloor}_{\bf i}([\varepsilon_n, 1-\varepsilon_n])\subset J\}\big)\geq \frac{1}{5}.
$$

\end{lem}

\begin{proof}
Let $\alpha, \delta, \varepsilon, M$ be the positive constants given in Lemma \ref{constants}. 
From Lemma \ref{constants} it follows that there exists $n_0\in\mathbb{N}$ such that for $n\ge n_0$ we have
$$
\mathbb{P}\big(\{{\bf i}\in\Sigma: \exists_{k\leq \lfloor\sqrt[4]{n}\rfloor} f_{\bf i}^k(\varepsilon_n)\geq\varepsilon \}\big)\geq\frac{9}{10}.
$$
From (\ref{delta_in_class}) and the fact that $M=\varepsilon^{-\alpha}<\frac{a^{-\alpha}}{6}$ we obtain
$$
\begin{aligned}
\mathbb{P}(\{{\bf i}\in\Omega : f_{\bf i}^{\lfloor\sqrt[4]{n}\rfloor}(\varepsilon_n)\geq a\})
&\geq \mathbb{P}\big(\{{\bf i}\in\Sigma: \exists_{k\leq \lfloor\sqrt[4]{n}\rfloor} f_{\bf i}^k(\varepsilon_n) \geq \varepsilon\}\big)\cdot(1-2Ma^\alpha)\\
&\geq \frac{9}{10}\cdot\frac{2}{3}=\frac{3}{5}.
\end{aligned}
$$
In the same way we may prove that for $n_0$ sufficiently large we have
$$
\mathbb{P}(\{{\bf i}\in\Omega : f_{\bf i}^{\lfloor\sqrt[4]{n}\rfloor}(1-\varepsilon_n)\leq 1-a\})\geq \frac{3}{5}\qquad\text{for $n\ge n_0$}.
$$
Combining these two estimates we obtain that there exists $n_0\in\mathbb N$ such that
$$
\mathbb{P}\big(\{{\bf i}\in\Sigma:  f^{\lfloor\sqrt[4]{n}\rfloor}_{\bf i}([\varepsilon_n, 1-\varepsilon_n])\subset J\}\big)\geq 1-\left(\frac 2 5 + \frac 2 5\right)=\frac{1}{5}\quad\text{for $n\ge n_0$.}
$$
The proof is complete.
\end{proof}


\begin{df}
 Let $\mathcal A\subset \Sigma\cup\Sigma_*$. We say that a sequence ${\mathbf i}\in\Sigma_*$ is dominated by $\mathcal A$ if there exists ${\mathbf j}\in\mathcal A$ such that ${\mathbf i}\prec{\mathbf j}$. Additionally, we shall assume that the empty sequence is dominated by any $\mathcal A$.
\end{df}

\begin{lem}
\label{dominating}
Let $\mathcal A\subset \Sigma$ be such that $\mathbb P(\mathcal A)\ge\beta$ for some $\beta>0$ and let $k, n\in\mathbb N$ with $k<n$. Then there exists a set $A\subset\Sigma_n$ such that $\mathbb P_n(\Sigma_n\setminus A)\le (1-\beta)^k$ and for any ${\mathbf i}\in A$ there exist ${\mathbf i}_1, {\mathbf i}_2,\ldots, {\mathbf i}_k\in\Sigma_*$ such that 
$$
{\mathbf i}={\mathbf i}_1{\mathbf i}_2\cdots{\mathbf i}_k
$$
and for $j=1,\ldots, k$ at least one of the sequences $\textbf i_j, \sigma\textbf i_j\ldots, \sigma^{k-1}\textbf i_j$, is dominated by $\mathcal A$.
\end{lem}

\begin{proof} Let $k, n\in\mathbb N$, $k<n$ and  $\mathcal A\subset \Sigma$ with $\mathbb P(\mathcal A)>0$ be given. Let $\beta>0$ be such that  $\mathbb P(\mathcal A)\ge\beta$. 
Define 
$$
B_1:=\{{\mathbf i}\in\Sigma_n: {\mathbf i}\text{ is dominated by}\,\, \mathcal A\}
$$
and observe that $\mathbb P_n(\Sigma_n\setminus B_1)\le 1-\beta$.

By $C_2$ denote the set of all ${\mathbf i}\in\Sigma_*$ such that ${\mathbf i}$ is not dominated by $B_1$ but ${\mathbf i}_{| |{\mathbf i}|-1}$ is dominated. If $C_2$ is empty, then putting $\textbf i _m$ to be the empty sequences for $m=2,\ldots, k$ will finish the proof. In the other case define
$$
B_2:=\{{\mathbf i}{\mathbf j}_{|n}: {\mathbf i}\in C_2, {\mathbf j}\text{ is dominated by}\,\, \mathcal A\},
$$
where ${\mathbf i}{\mathbf j}_{|n}$ denotes the natural projection on $\Sigma_n$ of the concatenation of ${\mathbf i}$ and ${\mathbf j}$. 
Observe that $\mathbb P_n( B_2 |\Sigma_n\setminus B_1)>\beta$.
This gives $\mathbb{P}_n(\Sigma_n \setminus (B_1\cup B_2))=\mathbb{P}_n(\Sigma_n \setminus B_1)\cdot \mathbb P_n(\Sigma_n \setminus B_2 | \Sigma_n \setminus B_1)=\mathbb{P}_n(\Sigma_n \setminus B_1)(1 - \mathbb P_n( B_2 |\Sigma_n\setminus B_1))\leq (1-\beta)^2$. Obviously, if ${\mathbf i}\in B_2$, then ${\mathbf i}={\mathbf i}_1{\mathbf i}_2$ and for $j=1, 2$ either ${\mathbf i}_j$ or $\sigma{\mathbf i}_j$ is dominated by $\mathcal A$.

Further, if $B_1,\ldots, B_{l-1}$ are given for some $l\le k$, $\mathbb P_n(\Sigma_n\setminus(B_1\cup\cdots\cup B_{l-1}))\le (1-\beta)^{l-1}$ and if ${\mathbf i}\in B_m$ for $m=1,\ldots, l-1$, then ${\mathbf i}={\mathbf i}_1\cdots{\mathbf i}_m$ and  at least one of the sequences $\textbf i_j, \sigma\textbf i_j,\ldots, \sigma^{l-2}\textbf i_j$ is dominated by $\mathcal A$ for $j=1,\ldots, m.$ Define now the set $C_l$ of all $\textbf i\in\Sigma_*$ such that for some $m=1,\ldots, l$ the sequences $\textbf i, \textbf i_{||i|-1},\ldots, \textbf i_{||i|-m+1}$ are not dominated by $B_{l-1}$, but $\textbf i_{||i|-m}$ is dominated. If $C_l$ is empty, then putting $\textbf i _m$ to be the empty sequence for $m=l,\ldots, k$ will finish the proof.  In the other case define
$$
B_l:=\{{\mathbf i}{\mathbf j}_{|n}: {\mathbf i}\in C_l, {\mathbf j}\text{ is dominated by} \mathcal A\}.
$$
Observe that
$$
\begin{aligned}
&\mathbb P_n(\Sigma_n\setminus (B_1\cup\cdots\cup B_l))\\
&=\mathbb P_n(\Sigma_n\setminus (B_1\cup\cdots\cup B_{l-1}))\mathbb P_n(\Sigma_n\setminus B_l \ | \ \Sigma_n \setminus (B_1\cup\cdots\cup B_{l-1}))\leq (1-\beta)^l.
\end{aligned}
$$
Moreover, from the definition of $B_l$ it follows that
if ${\mathbf i}\in B_l$, then ${\mathbf i}={\mathbf i}_1\cdots{\mathbf i}_l$
and at least one of the sequences $\textbf i_m, \sigma\textbf i_m,\ldots, \sigma^{l-1}\textbf i_m$ for $m=1,\ldots, l$ is dominated by $\mathcal A$. 

Taking $l=k$ and setting $A=B_1\cup\cdots\cup B_k$ we finish the proof. Indeed, if ${\mathbf i}\in A$ there exist ${\mathbf i}_1, {\mathbf i}_2,\ldots, {\mathbf i}_k\in\Sigma_*$, some of them possibly empty sequences, such that 
$
{\mathbf i}={\mathbf i}_1{\mathbf i}_2\cdots{\mathbf i}_k
$
and for $j=1,2,\ldots, k$ at least one of the sequences $\textbf i_j, \sigma\textbf i_j,\ldots, \sigma^{k-1}\textbf i_j$ is dominated by $\mathcal A$.
The proof is complete.
\end{proof}

\section{The proof of the Central Limit Theorem}

Let $(f_1,\ldots, f_N; p_1,\ldots, p_N)$ be an admissible iterated function system and let $P$ be the corresponding Markov operator. Let $(X_n)$ be the Markov chain corresponding to $P$. 
This part of the paper is devoted to the proof of the quenched central limit theorem for the random process $(\varphi(X_n))$, where $\varphi$ is a Lipschitz function. The~question of the quenched central limit theorem was raised, for reversible Markov chains, by C. Kipnis and S.R.S. Varadhan in \cite{KV}. M.D. Gordin and B.A. Lif\v{s}hits proved that if $\mu_*$ is an ergodic invariant measure for a Markov operator $P$ and $\varphi$ is an $L_2(\mu_*)$ coboundary, i.e., $\varphi\in (I-U)L_2(\mu_*)$, then the quenched central limit theorem holds for $\mu_*$--a.e. $x$ (see \cite{BorIbra}). It is worth mentioning here that in our case the support of $\mu_*$ could be a nontrivial subset of the interval $[0, 1]$. In fact, for some admissible iterated function systems the unique invariant measure $\mu_*$ is distributed on a Cantor set (see \cite{BarSpiewak}). Recently it was proved that reversible chains satisfy a quenched invariance principle for $\varphi\in (I-U)L_q(\mu_*)\cap L_p(\mu_*)$, where $1\le q\le 2$ and $p=q/(q-1)$ (see \cite{BPP}). 
Since $\varphi$ in our considerations is a bounded function to prove the quenched central  limit theorem it is enough to check that $\varphi\in (I-U)L_q(\mu_*)$ for some $1\le q\le 2$. Unfortunately, we have failed to do it. Indeed, we are unable to verify that $\sum_{n=1}^{\infty} U^n\varphi (x)$ is convergent $\mu_*$--a.s for $\varphi$ such that $\langle \mu_*, \varphi\rangle=0$. This holds true when the Markov operator $P$ satisfies the spectral gap property in the Wasserstein--Kantorovich metric. Although our system is asymptotically stable on $(0, 1)$, its rate of convergence has not been yet determined. Our proof is therefore based on the Maxwell--Woodroofe theorem (see \cite{Maxwell-Woodroofe}). In fact, this theorem allows us to prove the annealed central limit theorem. On the other hand, to prove the quenched central limit theorem for every $x\in (0, 1)$ we construct some coupling between trajectories 
and evaluate the distance between Fourier transforms for Markov processes starting at different initial points.
\vskip3mm
We are now in a position to formulate and prove the main results of our paper.

\begin{thr}\label{T1_22.06.19}[Central Limit Theorem]
Let $(f_1,\ldots, f_N; p_1,\ldots, p_N)$ be an admissible iterated function system, and let $(X_n)$ be the corresponding stationary Markov chain with the initial distribution $\mu_*$. If $\varphi:[0,1]\rightarrow\mathbb{R}$ is a Lipschitz function with $\int_{[0, 1]}\varphi d\mu_*=0$, then the random process $(\varphi(X_n))$ satisfies the central limit theorem, i.e., the limit
$$
\sigma^2:=\lim_{n\to\infty} \mathbb{E}\bigg( \frac{\varphi(X_0)+\cdots+\varphi(X_n)}{\sqrt{n}} \bigg)^2
$$
exists and
$$\frac{\varphi(X_0)+\cdots+\varphi(X_n)}{\sqrt{n}}\Rightarrow \mathcal{N}(0,\sigma)\qquad\text{as $n\to\infty$,}
$$
where $\Rightarrow$  denotes convergence in distribution.
Moreover, the same is true for the process $(\varphi(X^x_n))$, where $(X^x_n)$ is the corresponding Markov chain starting from an arbitrary point $x\in (0,1)$.
\end{thr}

\begin{proof} Since $\mu_*$ is ergodic for $P$ by the uniqueness in Theorem 1, the chain is ergodic. Therefore to prove the first part of our theorem 
it is sufficient to show the following condition (see Theorem 1 in \cite{Maxwell-Woodroofe}):
\begin{equation}\label{e1_11.06.19}
\sum_{n=1}^\infty n^{-\frac{3}{2}}\big\| \sum_{j=1}^n U^j\varphi \big\|_{L^2(\mu_*)}<\infty,
\end{equation}
where $\|\cdot\|_{L^2(\mu_*)}$ denotes the $L^2-$norm with respect to the invariant measure $\mu_*$.

Let $\alpha, \delta, \varepsilon, M$ be the constants given in Lemma \ref{constants} and let $a$ be such that $M=\varepsilon^{-\alpha}<\frac{a^{-\alpha}}{6}$. Choose $n_0\in\mathbb N$ according to Lemma \ref{return_to_J}.
Put $J=[a,1-a]$ and define
$$
E_n:=\{\textbf{i}\in \Sigma: f_\textbf{i}^{\lfloor\sqrt[4]{n}\rfloor}([\varepsilon_n, 1-\varepsilon_n])\subset J\}\qquad\text{for $n\ge n_0$}.
$$
From Lemma \ref{return_to_J} we have
$$
\mathbb{P}\big(E_n\big)\geq \frac{1}{5}\qquad\text{for $n\ge n_0$}.
$$
Let $\Omega\subseteq\Sigma$ be the set given in Lemma \ref{omega} and define
$$
\mathcal{A}=\{\textbf{i}\textbf{j}_{|n}: \textbf{i}\in E_n, \textbf{j}\in\Omega\}\subseteq \Sigma_n.
$$
It is easily seen that $\mathbb{P}_n(\mathcal{A})\geq \beta>0$ for a certain $\beta$ independent of $n$. We apply Lemma \ref{dominating} to the set $\mathcal{A}$, $k=\lfloor\sqrt[8]{n}\rfloor$ for $n\ge n_0$ and obtain some sequence of sets $A_n\subseteq\Sigma_n$ with $\mathbb{P}_n(\Sigma_n\setminus A_n)\leq (1-\delta)^{\lfloor\sqrt[8]{n}\rfloor}$ such that if ${\bf i}\in A_n$, then $\textbf{i}=\textbf{i}_1\cdots\textbf{i}_k$ and for every $m=1,\ldots, k$ at least one of the sequences $\textbf i_m, \sigma\textbf i_m,\ldots, \sigma^{k-1}\textbf i_m$ is dominated by $\mathcal A$. Hence for $\textbf{i}_m$, $m=1,\ldots, k$ and for $x, y\in [\varepsilon_n, 1-\varepsilon_n]$ we have
$$\sum_{j=1}^{|\textbf{i}_m|} \big| f_\textbf{i}^j(x)- f_\textbf{i}^j(y) \big|\leq \lfloor\sqrt[4]{n}\rfloor+r+\lfloor\sqrt[8]{n}\rfloor+\frac{q}{1-q}\leq 2\bigg(\lfloor\sqrt[4]{n}\rfloor+r+\frac{q}{1-q}\bigg),
$$
where $r$ is the constant given in Lemma \ref{omega}. Further, set
$$D_n:=\{\textbf{i}\in A_n: \exists_{m\geq \lfloor\sqrt[4]{n}\rfloor} f_\textbf{i}^m(x)<\varepsilon_n \ \textrm{or}\\ \ f_\textbf{i}^m(y)>1-\varepsilon_n\\ \}$$
and notice that for any sequence $\textbf{i}\in A_n\setminus D_n$ and every $x, y\in [\varepsilon_n, 1-\varepsilon_n]$ we have
\begin{equation}\label{e1_9.12.18}
\sum_{j=1}^n \big| f_\textbf{i}^j(x)- f_\textbf{i}^j(y) \big|\leq k\left(\lfloor\sqrt[4]{n}\rfloor+r+\frac{q}{1-q}\right)\leq C_1n^{\frac{3}{8}}
\end{equation}
for some positive constant $C_1$. Furthermore, $\mathbb{P}_n(D_n)\leq 2M\gamma_n$, by Lemma \ref{mass_boundary}.

Let $B_n:=\Sigma_n\setminus A_n$. From Lemma \ref{dominating} we obtain that $\mathbb{P}_n(B_n)\leq (1-\beta)^{\lfloor\sqrt[8]{n}\rfloor}$. Denote by $L$ the Lipschitz constant of $\varphi$. Since $\int_{[0, 1]}\varphi\d\mu_*=\int_{[0, 1]} U^j\varphi\d\mu_*=0$ for $j\in\mathbb N$, we have
$$
\begin{aligned}
\int_{[0, 1]}\big|\sum_{j=1}^n U^j\varphi(x)\big|^2\mu_*(\d x)&=\int_{[0, 1]}\bigg|\int_{[0, 1]}\sum_{j=1}^n (U^j\varphi(x)-U^j\varphi(y))\mu_*(\d y)\bigg|^2\mu_*(\d x)\\
&\leq \int_{[0, 1]}\int_{[0, 1]}\bigg( \sum_{j=1}^n |U^j\varphi(x)-U^j\varphi(y)| \bigg)^2\mu_*(\d x)\mu_*(\d y).
\end{aligned}
$$
From the fact that $\mu_*([0,x])\leq Mx^\alpha$ and $\mu_*([1-x,1])\leq Mx^\alpha$ for every $x\in[0,1]$ we obtain the following inequality
$$
\begin{aligned}
&\int_{[0, 1]}\int_{[0, 1]}\bigg( \sum_{j=1}^n |U^j\varphi(x)-U^j\varphi(y)| \bigg)^2\mu_*(\d x)\mu_*(\d y)\\
&\leq\int_{[{\varepsilon_n}, {1-\varepsilon_n}]}\int_{[{\varepsilon_n}, {1-\varepsilon_n}]}\bigg( \sum_{j=1}^n |U^j\varphi(x)-U^j\varphi(y)| \bigg)^2\mu_*(\d x)\mu_*(\d y)+16\varepsilon_n^\alpha n^2 M\|\varphi\|^2.
\end{aligned}
$$
Since $\varepsilon_n=(1-\delta)^{\frac{1}{2}\lfloor\sqrt[4]{n}\rfloor}$, the last sequence is bounded by, say, $C_2>0$. However, by estimate (\ref{e1_9.12.18}) for $x, y\in [\varepsilon_n, 1-\varepsilon_n]$  we have 
$$
\begin{aligned}
&\sum_{j=1}^n |U^j\varphi(x)-U^j\varphi(y)|\leq \int_{A_n\setminus D_n}\sum_{j=1}^n|\varphi(f_{\bf i}^j(x))-\varphi(f_{\bf i}^j(y))|\textrm{d}\mathbb{P}_n({\bf i})\\
&+ \int_{D_n}\sum_{j=1}^n|\varphi(f_{\bf i}^j(x))-\varphi(f_{\bf i}^j(y))|\textrm{d}\mathbb{P}_n({\bf i})+ \int_{B_n}\sum_{j=1}^n|\varphi(f_{\bf i}^j(x))-\varphi(f_{\bf i}^j(y))|\textrm{d}\mathbb{P}_n({\bf i})\\
&\leq LC_1n^{\frac 3 8}+2nL\mathbb{P}_n(D_n)+2nL\mathbb{P}_n(B_n).
\end{aligned}
$$
By the fact that $\mathbb{P}_n(B_n)\leq (1-\beta)^{\lfloor\sqrt[8]{n}\rfloor}$ and $\mathbb{P}_n(D_n)\leq 2M\gamma_n$, the last two sequences above are bounded and therefore there exists a positive constant $C_3$ such that
$$
\sum_{j=1}^n |U^j\varphi(x)-U^j\varphi(y)|\leq C_3 n^{\frac 3 8}\qquad\text{for $x, y\in [\varepsilon_n, 1-\varepsilon_n]$.}
$$
Combining the above estimates we finally obtain
$$
\int_{[0, 1]}\int_{[0, 1]}\bigg( \sum_{j=1}^n |U^j\varphi(x)-U^j\varphi(y)| \bigg)^2\mu_*(\d x)\mu_*(\d y)\leq C_3^2n^{\frac 6 8}+C_2.
$$
Therefore there exists a positive constant $C$  such that
$$
\big\|\sum_{j=1}^n U^j\varphi\big\|_{L^2(\mu_*)}\leq Cn^{\frac{3}{8}}.
$$
This proves (\ref{e1_11.06.19}). Application of Theorem 1 in \cite{Maxwell-Woodroofe} finishes the proof for the stationary sequence $(X_n)$.

To derive the central limit theorem for $(X^x_n)$, $x\in (0,1)$, observe that by conditions (\ref{e120_23.09.19}) and (\ref{e121_23.09.19}) the characteristic functions of the processes
$$
\frac{1}{\sqrt{n}}(\varphi\big(X_0)+\cdots+\varphi(X_n)\big) \ \mathrm{and} \ \frac{1}{\sqrt{n}}(\varphi\big(X^x_0)+\cdots+\varphi(X^x_n)\big)$$
are of the form, respectively,
$$
\Phi_n(t)=\int_{[0, 1]}\int_{\Sigma_n} \exp\bigg(it\frac{\varphi(f^1_{\bf i}(y))+\cdots+\varphi(f^n_{\bf i}(y))}{\sqrt{n}}\bigg)\mathbb{P}_n(\d {\bf i})\mu_*(\d y) \ \mathrm{for} \ t\in\mathbb{R}
$$
and
$$
\Phi^x_n(t)=\int_{\Sigma_n} \exp\bigg(it\frac{\varphi(f^1_{\bf i}(x))+\cdots+\varphi(f^n_{\bf i}(x))}{\sqrt{n}}\bigg)\mathbb{P}_n(\d {\bf i}) \ \mathrm{for} \ t\in\mathbb{R}.
$$
From the first part of the proof we know that there exists $\sigma\geq 0$ such that for every $t\in\mathbb R$
$$
\Phi_n(t)\to \exp\bigg(\frac{-t^2\sigma^2}{2}\bigg)\qquad\text{as $n\to\infty$.}
$$
Therefore, to finish the proof it is sufficient to show for every $t\in\mathbb R$ the convergence
$$
\big|\Phi^x_n(t)-\Phi_n(t)\big|\to 0\qquad\text{as $n\to\infty$}.
$$

Since $|e^{itx_1}-e^{itx_2}|\leq |t||x_1-x_2|$, for every $t\in\mathbb R$, for sufficiently large $n\in\mathbb N$ and $x\in [\varepsilon_n, 1-\varepsilon_n]$ using $\mu_*\in\mathcal P_{M, \alpha}$ we have
$$
\begin{aligned}
&\big|\Phi^x_n(t)-\Phi_n(t)\big| \le\frac{|t|}{\sqrt{n}}\cdot\int_{[0, 1]}\int_{\Sigma_n}\bigg| \sum_{j=1}^n(\varphi(f^j_{{\bf i}}(x))-\varphi(f^j_{{\bf i}}(y))\bigg|\mathbb P_n(\d{\bf i})
\mu_*(\d y)\\ 
&\leq \frac{|t|}{\sqrt{n}}\int_{[{\varepsilon_n}, {1-\varepsilon_n}]} \int_{\Sigma_n}\bigg| \sum_{j=1}^n(\varphi(f^j_{{\bf i}}(x))-\varphi(f^j_{{\bf i}}(y))\bigg|\mathbb P_n(\d{\bf i})
\mu_*(\d y) +\frac{|t|}{\sqrt{n}}\cdot 2n\|\varphi\| M\varepsilon_n^\alpha\\
&\leq \frac{|t|}{\sqrt{n}}\int_{[{\varepsilon_n}, {1-\varepsilon_n}]}\int_{A_n\setminus D_n}\bigg| \sum_{j=1}^n(\varphi(f^j_{{\bf i}}(x))-\varphi(f^j_{{\bf i}}(y))\bigg|\mathbb P_n(\d{\bf i})
\mu_*(\d y)\\
& +\frac{|t|}{\sqrt{n}}\cdot 2n\|\varphi\| M\varepsilon_n^\alpha + \frac{|t|}{\sqrt{n}}\cdot 2n\|\varphi\| (\mathbb P_n(\Sigma_n\setminus A_n) +\mathbb P_n(D_n)).
\end{aligned}
$$
Since $\varepsilon_n\to 0$ as $n\to\infty$, for every $x\in (0, 1)$ and $t\in\mathbb R$, by (\ref{e1_9.12.18}) and using the estimates $\mathbb P_n(B_n)\le (1-\beta)^{\lfloor\sqrt[8]{n}\rfloor}$ and $\mathbb P_n(D_n)\le 2M\gamma_n$,
 we finally obtain
 $$
 \begin{aligned}
&\lim_{n\to\infty} |\Phi^x_n(t)-\Phi_n(t)| \leq\lim_{n\to\infty} \frac{2|t|}{\sqrt{n}}\bigg(C_1 Ln^{\frac 3 8}+n\|\varphi\|(1-\beta)^{\lfloor\sqrt[8]{n}\rfloor}\bigg)\\
&+\lim_{n\to\infty} 2|t|\|\varphi\|\sqrt{n}\bigg((1-\delta)^{\lfloor\sqrt[4]{n}\rfloor}+M\varepsilon_n^\alpha\bigg)=0.
 \end{aligned}
 $$
This completes the proof.
\end{proof}

\begin{cor} If the assumptions of Theorem \ref{T1_22.06.19} are satisfied, then the stationary sequence
$
(\frac{1}{\sqrt{n}}\sum_{k=0}^{[nt]}\varphi(X_k))
$
converges in distribution to $\sigma W(t)$, where $W(t)$, $0\le t\le 1$ is a Brownian motion. Moreover, for $\mu_*$-a.e. $x\in (0, 1)$ the sequence  $
(\frac{1}{\sqrt{n}}\sum_{k=0}^{[nt]}\varphi(X_k^x))
$
also converges in distribution to $\sigma W(t)$, $0\le t\le 1$.
\end{cor}
\begin{proof} From the proof of Theorem \ref{T1_22.06.19}  it follows that the Maxwell--Woodroofe condition
$$
\sum_{n=1}^\infty n^{-\frac{3}{2}}\big\| \sum_{j=1}^n U^j\varphi \big\|_{L^2(\mu_*)}<\infty
$$
is satisfied. Applying Theorem 1.1 in \cite{PelUtev} we obtain the first assertion. On the other hand, the second assertion follows from Theorem 2.7 in \cite{CunyMer}.\,\,\,\end{proof}                                                                                                                                                                                                                                                                                                                                                                                                    
\vskip3mm
{\bf Acknowledgments.} The authors wish to express their gratitude to an anonymous
referee for thorough reading of the manuscript and valuable remarks.


\bibliographystyle{plain}
\bibliography{CLT_22.08.19}

\end{document}